\pdfoutput=1

\documentclass[letterpaper, 11pt]{article}
\usepackage{amssymb,amsmath,amsthm}
\usepackage{epsf}
\usepackage{times,bm,mathrsfs}
\usepackage{amsmath}
\usepackage{amssymb}
\usepackage{amsfonts}
\usepackage{mathrsfs}
\usepackage{bbm}
\usepackage{color}
\usepackage{graphicx}
\usepackage{amscd}
\usepackage{epsfig}
\usepackage{fullpage}
\usepackage{times}

\definecolor{Red}{rgb}{1,0,0}
\definecolor{Blue}{rgb}{0,0,1}
\definecolor{Olive}{rgb}{0.41,0.55,0.13}
\definecolor{Green}{rgb}{0,1,0}
\definecolor{MGreen}{rgb}{0,0.8,0}
\definecolor{DGreen}{rgb}{0,0.55,0}
\definecolor{Yellow}{rgb}{1,1,0}
\definecolor{Cyan}{rgb}{0,1,1}
\definecolor{Magenta}{rgb}{1,0,1}
\definecolor{Orange}{rgb}{1,.5,0}
\definecolor{Violet}{rgb}{.5,0,.5}
\definecolor{Purple}{rgb}{.75,0,.25}
\definecolor{Brown}{rgb}{.75,.5,.25}
\definecolor{Grey}{rgb}{.5,.5,.5}
\definecolor{Black}{rgb}{0,0,0}

\newcommand{\eps}{\varepsilon}

\newcommand{\ind}{\mathbf{1}}

\newcommand{\E}{\mathbb{E}}
\renewcommand{\P}{\mathbb{P}}

\newcommand{\diff}{\mathrm{d}}

\newtheorem{theorem}{Theorem}
\newtheorem{proposition}{Proposition}
\newtheorem{corollary}{Corollary}
\newtheorem{definition}{Definition}
\newtheorem{example}{Example}
\newtheorem{lemma}{Lemma} 

\newcommand{\weight}{w}

\newcommand{\tree}{\mathcal{T}}
\newcommand{\edges}{\mathcal{E}}
\newcommand{\vertices}{\mathcal{V}}
\renewcommand{\root}{\rho}
\newcommand{\state}{\sigma}
\newcommand{\estate}{\hat{\state}}
\renewcommand{\weight}{\theta}
\newcommand{\minweight}{\weight_*}

\newcommand{\cutset}{\pi}
\newcommand{\allcutsets}{\mathscr{C}}
\newcommand{\fitchstate}{\widehat{S}}
\newcommand{\accuracy}{\mathrm{RA}}

\newcommand{\FF}{\mathbf{u}}
\newcommand{\GG}{\mathbf{d}}
\newcommand{\FFabs}{\mathbf{U}}
\newcommand{\distr}{\Theta}
\newcommand{\tweight}{\tilde{\weight}}
\newcommand{\tperc}{\tilde{J}}
\newcommand{\ttree}{\tilde{\tree}}
\newcommand{\tdistr}{\widetilde{\Theta}}
\newcommand{\tedges}{\tilde{\edges}}
\newcommand{\tvertices}{\tilde{\vertices}}
\newcommand{\noext}{\mathcal{N}_{\mathrm{ext}}}
\newcommand{\cutoff}{\minweight}
\newcommand{\pcut}{\tau}
\newcommand{\neighb}{\mathcal{D}}

\addtocounter{page}{-1}

\begin{document}

\title{\vspace{0cm}
Sufficient condition
for 
root reconstruction
by parsimony
on binary trees with general weights
\footnote{
Keywords: 
Markov model on a tree,
reconstruction problem,
branching number,
parsimony principle}
}

\author{
Sebastien Roch\footnote{Department of Mathematics at the University of Wisconsin--Madison.
	Supported by NSF grants DMS-1248176, DMS-1149312 (CAREER), DMS-1614242, CCF-1740707 (TRIPODS), DMS-1902892, DMS-2023239, as well as a Simons Fellowship
	and a Vilas Associates Award. \texttt{roch@math.wisc.edu}}
\and
Kun-Chieh Wang
}
\maketitle

\begin{abstract}
We consider the problem of inferring an ancestral state from observations at the leaves of a tree, assuming
the state evolves along the tree according to a two-state symmetric Markov process. We establish a general branching rate condition
under which maximum parsimony, a common reconstruction method requiring only the knowledge of the tree, succeeds better than random guessing uniformly in the depth of the tree. We thereby generalize previous results of~\cite{GascuelSteel:10,ZhShYa+:10}. Our results
apply to both deterministic and i.i.d.~edge weights.
\end{abstract}

\thispagestyle{empty}

\clearpage

\section{Introduction}\label{section:introduction}

\paragraph{Ancestral reconstruction}
In biology, the inferred evolutionary history of organisms is depicted by a phylogenetic tree, that is, a rooted tree whose branchings indicate past speciation events with the leaves representing living species. The evolution
of features, such as the nucleotide at a given position in the genome of a species, the presence or absence of a protein or the number of horns in a lizard, is commonly assumed to follow Markovian dynamics along the tree~\cite{Steel:16}. That is, on each edge, the state of the feature changes according to a Markov process; at bifurcations, two independent copies of the feature evolve along the outgoing edges starting from the state at the branching point.

Here we consider the problem of inferring an ancestral state from observations of a feature at the leaves of a known phylogenetic tree. We refer to this problem, which has important applications in biology~\cite{Mossel:04a,DaMoRo:11a,roch_phase_2017}, as the ancestral reconstruction problem. Many rigorous results have been obtained in previous work; see, e.g.,~\cite{KestenStigum:66,BlRuZa:95,Ioffe:96a,TuffleySteel:97,Mossel:98,EvKePeSc:00,Mossel:01,MosselPeres:03,BoChMoRo:06,GascuelSteel:10,ZhShYa+:10,Sly:11,mossel_robust_2013,mossel_majority_2014,bhatnagar_decay_2016,liu_tightness_2018,fan_necessary_2018,makur_broadcasting_2020,fan_statistically_2020,gu_broadcasting_2020,addario-berry_broadcasting_2020} for a partial list. Typically, one seeks an estimator of the root state which is strictly superior to random guessing---uniformly in the depth of the tree---under a uniform prior on the root~\cite{Mossel:01}. Whether such an estimator exists has been shown to depend on a trade-off between the mixing rate of the Markov process and the growth rate of the tree. In some cases, for instance in two-state symmetric Markov chains on binary trees~\cite{KestenStigum:66,Ioffe:96a} as well as on more general trees~\cite{EvKePeSc:00}, sharp thresholds have been established. We focus on that special case here.

\vspace{-0.4cm}\paragraph{The threshold of maximum parsimony}
Maximum parsimony is an ancestral reconstruction method with a long history in evolutionary biology. See e.g.~\cite{Steel:16}. Its underlying principle is simple and intuitive: it calls for assigning a state to {\em each} internal vertex in such a way as to {\em minimize the total number of changes} along the edges; the resulting state at the {\em root} (possibly not unique) is the desired reconstructed state. An advantage of maximum parsimony, which accounts partly for its occasional use in practice, is that it only requires knowledge of the tree---not of the substitution probabilities along the edges, which can be difficult to estimate accurately from data
and can affect the accuracy of reconstruction~\cite{cusimano_ultrametric_2014}. In fact maximum parsimony is equivalent to maximum likelihood under those circumstances~\cite{TuffleySteel:97}. Another root state estimator with that property is recursive majority, studied in~\cite{Mossel:98,Mossel:04a}. 

The theoretical properties of maximum parsimony for ancestral state reconstruction have been widely studied~\cite{Steel:89,TuffleySteel:97,fischer_maximum_2009,ZhShYa+:10,GascuelSteel:10,chai_number_2011,gascuel_predicting_2014,herbst_ancestral_2017,herbst_accuracy_2018,herbst_quantifying_2019}. From its very definition, 
one might expect maximum parsimony to perform well 
when the probabilities of substitution along the edges are ``sufficiently small.'' This statement was confirmed rigorously for the two-state symmetric model on complete binary trees with constant mutation probability in~\cite{ZhShYa+:10} (as first conjectured in~\cite{Steel:89}). Below a critical probability, maximum parsimony does indeed beat random guessing. A related result was obtained in~\cite{GascuelSteel:10} under a common random tree model known as the Yule tree. Here we substantially generalize both of these results. We give a general sufficient condition on the branching rate under which maximum parsimony 
succeeds. Our condition is related to the branching number, which plays an important in many processes on trees. See e.g.~\cite{LyonsPeres:16}. Other generalized notions of 
degree have been used in related contexts~\cite{sinclair_spatial_2017}.

\vspace{-0.4cm}\paragraph{Definitions}
In order to state our results formally, we begin with some definitions.
Let $\tree = (\vertices, \edges)$ be an infinite complete binary tree
rooted at $\root$. That is, all vertices of $\tree$ have exactly two children; in particular, $\tree$ has no leaf. Every edge $e \in \edges$ is assigned a weight $\weight_e \in [0,1]$. If $e = (x,y)$ where $x$ is the parent of $y$,
we also write $\weight_y = \weight_e$. 
We use the notation $x \leq y$ 
to indicate that $x$ is an ancestor of $y$ and we write
$x < y$ for $x \leq y$ and $x \neq y$. We also
let $s(x)$ be the sibling of $x \neq \root$.

Under the {\em Cavender-Farris (CF) model} 
over $\tree$ and $\weight = (\weight_y)_{y\in \vertices}$,
also known as Neyman $2$-state model,
we associate to each vertex $x \in \vertices$ a
state $\state_x \in \{0,1\}$ as follows. The state at the
root, $\state_\root$, is picked uniformly at random
in $\{0,1\}$. Recursively, if $y$ has parent $x$,
state $\state_y$ is equal to $\state_x$ with probability $\weight_y$, otherwise it is picked uniformly at random
in $\{0,1\}$. We let 
$p_y = \P_{\tree,\weight}[\state_x \neq \state_y]
= \frac{1-\weight_y}{2}$
be the {\em probability of a substitution} on
edge $e = (x,y)$.
Here $\P_{\tree,\weight}$ stands for the probability 
operator under the CF model over $\tree$ and
$\weight$.
The CF model is equivalent to
the Fortuin-Kastelyn random cluster representation
of the ferromagnetic Ising model on $\tree$ with
a free boundary. See e.g.~\cite{EvKePeSc:00}
and references therein.

Informally,
a root state estimator is a map
returning a (possibly randomized) guess 
for the state of the root, based on the knowledge of
the states on an observed cutset.
\begin{definition}[Cutset]
	A {\em cutset} is a minimal, finite set of vertices $\cutset \subseteq \vertices$ such that all infinite self-avoiding paths starting at $\root$ must
	visit $\cutset$. We let $\allcutsets(\tree)$ be the 
	collection of all cutsets of $\tree$. We denote by
	$\tree^\cutset = (\vertices^\cutset,\edges^\cutset)$ the finite tree obtained from
	$\tree$ after removing all descendants of the
	vertices in $\cutset$.
\end{definition}
\noindent As mentioned above, our focus in this work is on a root state estimator
known in phylogenetics as {\em maximum parsimony}. 
The parsimony principle dictates that
one assigns to each vertex $x$ 
(ancestor to the observed cutset $\cutset$)
a state $\estate_x$ such that the overall number
of changes along the edges of $\tree^\cutset$, namely,
$\sum_{(x,y)\in\edges^\cutset} \ind\{\estate_x \neq \estate_y\}$,
is minimized,
where we let by default $\estate_z = \state_z$ 
for all $z \in \cutset$. In case both $0$ and $1$
can be obtained in this way as root state, 
a uniformly random value in $\{0,1\}$ is returned.
We let $\accuracy^\cutset_{\tree,\weight}$ be 
the {\em reconstruction accuracy} 
of parsimony, i.e.~the probability that it correctly reconstructs the root state.

\vspace{-0.4cm}\paragraph{Main result: deterministic weights}
In our main result, we give conditions under which the reconstruction
accuracy of maximum parsimony is uniformly bounded away from $1/2$. 
\begin{theorem}[Reconstruction accuracy of parsimony: deterministic weights]
	\label{theorem:deterministic}
Let $\tree = (\vertices,\edges)$ be an infinite complete binary tree with
edge weights $\weight = (\weight_z)_{z\in \vertices}$ satisfying
$\minweight := \inf_{z \in \vertices} \weight_z > 0$
and
\begin{equation}
\label{eq:branching-cond}
\sup\left\{\kappa > 0 : \inf_{\cutset \in \allcutsets(\tree)}
\sum_{x\in\cutset}\prod_{\root \neq z \leq x} \kappa^{-1}\weight_{z} > 0\right\}
> \frac{3}{2}.
\end{equation}
Then $\inf_{\cutset \in \allcutsets(\tree)}
\accuracy_{\tree,\weight}
^\cutset > \frac{1}{2}$, i.e., the reconstruction accuracy of parsimony
on $\tree$ is away from $1/2$.
\end{theorem}
\noindent 
Condition~\eqref{eq:branching-cond} involves the {\em branching number}, a generalized notion of branching rate which
plays a key role in the analysis of many stochastic
processes on trees and tree-like graphs. See e.g.~\cite{LyonsPeres:16}.
The following example provides some intuition in a special case.
\begin{example}[Fixed edge weights]
	\label{example:br-fixed-weights}
	As a simple illustration, observe that,
	when all weights are equal to $\minweight \in (0,1]$,
	the supremum in~\eqref{eq:branching-cond} is attained for $\kappa = 2\minweight$.
	Indeed,
	the sum in~\eqref{eq:branching-cond} when $\kappa = 2\minweight$
	simplifies to
	\begin{equation*}
	\sum_{x\in \cutset}
	\prod_{\root \neq z \leq x}
	(2 \minweight)^{-1}
	\weight_{z}
	= 
	\sum_{x\in \cutset}
	\prod_{\root \neq z \leq x}
	\frac{1}{2}  
	= 1 > 0
	\end{equation*} 
	for any cutset $\cutset$,
	where the equality can be proved by induction
	on the graph distance from the root to the furthest
	vertex in $\cutset$. On the other hand, letting
	$\cutset_n$ be the cutset of all vertices at graph
	distance $n$ from $\root$,
	for any $\eps > 0$
	it holds that
	\begin{equation*}
	\sum_{x\in \cutset_n}
	\prod_{\root \neq z \leq x}
	((2 + \eps)\minweight)^{-1}
	\weight_{z}
	= 
	(2 + \eps)^{-n} \cdot 2^n
	\to 0,
	\end{equation*} 
	as $n \to +\infty$. 
	Hence, in this case, condition~\eqref{eq:branching-cond} reduces to $2\minweight > 3/2$, that is, $\minweight > 3/4$. In terms of substitution probability, this is $p_* = \frac{1 - \minweight}{2} < 1/8$.
\end{example}
\noindent The argument in the example above leads
to the following corollary.
\begin{corollary}[All substitution probabilities below the threshold]
Let the substitution probabilities
$(p_z)_{z\in \vertices}$ satisfy $p_z \in [0,p_*]$ for all $z \in \vertices$, for some $p_* < 1/8$.
Then the reconstruction accuracy of maximum parsimony
on $\tree$ is bounded away from $1/2$.
\end{corollary}
\noindent In general, Theorem~\ref{theorem:deterministic} 
cannot be improved in the following sense. It was shown in~\cite[Theorem 4.1]{ZhShYa+:10} that when $p_z =p_* > 1/8$,
that is, $\weight_z = 1 - 2p_z = \minweight < 3/4$
for all $z$ then
$\inf_{n}
\accuracy_{\tree,\weight}
^{\cutset_n} = \frac{1}{2}$,
where $\cutset_n$ is defined in Example~\ref{example:br-fixed-weights}. On the other hand, it is not known whether the reconstruction accuracy necessarily converges to $1/2$ if~\eqref{eq:branching-cond} is not satisfied. We leave this as an open problem.

Zhang et al.~\cite{ZhShYa+:10} also established the special case
of Theorem~\ref{theorem:deterministic} when
$p_z =p_* < 1/8$ for all $z$ and $\cutset = \cutset_n$.
Their proof proceeds through a careful analysis 
of the limit of a recurrence for $\accuracy^{\cutset_n}_{\tree,\weight}$ first derived
in~\cite{Steel:89}. Our
more general result follows from a softer argument which
relies on the instability of a fixed point of this recurrence
corresponding to asymptotic reconstruction accuracy $1/2$. A more detailed proof sketch is given in Section~\ref{sec:prelim} following some preliminaries. We note that our proof method may be of more general interest, e.g.,~to extend the results beyond the two-state case where the higher dimensionality of the system may complicate significantly the derivation of an explicit limit even when edge weights are constant.

\vspace{-0.4cm}\paragraph{Main result: i.i.d.~weights}
We also obtain a related result in the case of edge weights that are i.i.d. No lower bound on the weights is needed in this case, unlike Theorem~\ref{theorem:deterministic}.
\begin{theorem}[Reconstruction accuracy of parsimony: i.i.d.~weights]
	\label{theorem:iid}
	Let $\tree = (\vertices,\edges)$ be an infinite complete binary tree with
	edge weights $\weight = (\weight_z)_{z\in \vertices}$ drawn i.i.d.~from a
	distribution $\distr$ over $(0,1]$. Let $\mu$ be the mean of $\distr$ 
	and assume that $\mu > 3/4$. Then, for any $\delta > 0$, 
	there is $\eps > 0$ such that
$	\inf_{\cutset \in \allcutsets(\tree)}
	\accuracy_{\tree,\weight}
	^\cutset > \frac{1}{2}(1+\eps),
$
	with probability at least $1-\delta$.
\end{theorem}
\noindent The previous theorem covers in particular 
the case of the pure birth process, or Yule tree, which
is a popular random model of phylogenetic trees. See e.g.~\cite{Steel:16}. In that case,
$\weight_z = e^{-2 T_z}$, where $T_z$ is an exponential
with rate $\lambda$. To derive the corresponding threshold,
we note that
$$
\mu
= \E[\weight_z]
= \int_0^\infty e^{-2t} \lambda e^{-\lambda t}\,\diff t
= \frac{\lambda}{\lambda + 2}.
$$
Then
$
\mu > \frac{3}{4}
\iff
\lambda > 6,
$
which is consistent with the results of~\cite{GascuelSteel:10}.

\section{Preliminaries}
\label{sec:prelim}


\paragraph{Computing parsimony}
Our proofs are based on a recurrence for the reconstruction accuracy.
Maximum parsimony can be computed
efficiently by dynamic programming, which is referred to as the {\em Fitch method}.
\begin{definition}[Parsimony recursion]
	\label{definition:fitch}
	The Fitch method recursively
	constructs a set $\fitchstate^\cutset_z$ of possible states for each vertex $z \in \vertices^\cutset$, starting
	from $\pi$, as follows. If $z\in \pi$, $\fitchstate^\cutset_z = \{\state_z\}$. If $z\notin \pi$ and has children $x$ and $y$,
	\begin{eqnarray*}
		\fitchstate^\cutset_{z} 
		=
		\begin{cases}
			\fitchstate^\cutset_{x} \cap \fitchstate^\cutset_{y}, & \text{if $\fitchstate^\cutset_{x} \cap \fitchstate^\cutset_{y}\neq\emptyset$} \\
			\fitchstate^\cutset_{x} \cup \fitchstate^\cutset_{y}, & \text{o.w.}
		\end{cases}
	\end{eqnarray*}
	The method returns the maximum parsimony estimator
	$\estate_\root$ which is equal to the unique state
	in $\fitchstate^\cutset_\root$ if $|\fitchstate^\cutset_\root| = 1$,
	and otherwise returns a uniformly random value in $\{0,1\}$.
\end{definition}
\noindent What is described above is
the bottom-up phase of the Fitch method. (A top-down
phase, which we will not require here, then assigns a state
to each vertex in $\vertices^\cutset$ consistent
with a maximum parsimony solution;~e.g.~\cite{Steel:16}.)
Let $\cutset$ be an arbitrary cutset
on $\tree$
with states $\state_u$, $u\in \pi$,
and let $\fitchstate_z$, $z\in \vertices^\cutset$,
be the corresponding reconstructed sets under the Fitch method.
We define
\begin{eqnarray*}
	&&\alpha_z^\cutset = \P_{\tree,\weight}\left[\fitchstate^\cutset_z = \{\sigma_z\}\right], \qquad 
	\beta_z^\cutset = \P_{\tree,\weight}\left[\fitchstate^\cutset_z = \{1-\sigma_z\}\right].
\end{eqnarray*}
Under our randomization
rule, the reconstruction accuracy 
of maximum parsimony $\accuracy_{\tree,\weight}
^\cutset$ is given by
\begin{eqnarray}
	\P_{\tree,\weight}\left[\fitchstate^\cutset_\root = \{\sigma_\root\}\right]
	+ \frac{1}{2} \P_{\tree,\weight}\left[\fitchstate^\cutset_\root = \{0,1\}\right]
	&=& \alpha^\cutset_\root + \frac{1}{2}(1 - \alpha^\cutset_\root - \beta^\cutset_\root)
	= \frac{1}{2} + \frac{1}{2}(\alpha^\cutset_\root - \beta^\cutset_\root).\label{eq:ra-alphabeta}
\end{eqnarray}

\vspace{-0.4cm}\paragraph{Proof sketch}
Fix $\weight$ satisfying the assumptions
of Theorem~\ref{theorem:deterministic} and fix
a cutset $\cutset$. 
To analyze $(\alpha^\cutset_z,\beta^\cutset_z)$, 
it is natural to take advantage of
the recursive nature of $\tree$.
Let $x$ and $y$ be the children of $z$. The event
$\fitchstate^\cutset_z = \{\sigma_z\}$ occurs when
either (i) $\fitchstate^\cutset_x = \{\sigma_z\}$
and $\fitchstate^\cutset_y = \{\sigma_z\},$ or
(ii) $\fitchstate^\cutset_x = \{\sigma_z\}$
and $\fitchstate^\cutset_y = \{0,1\}$ or vice versa.
By the Markov property of the CF model, the random variables 
$\fitchstate^\cutset_x$ and $\fitchstate^\cutset_y$,
which are functions only of the states of $\pi$ under $x$ and
$y$ respectively,
are conditionally independent
given
$\state_z$. Hence, letting $q_u = 1 - p_u$ for $u = x,y$ and
taking into account the possibility of a mutation
along the edges $(z,x)$ and $(z,y)$, it follows 
as first derived in~\cite[Lemma 7.20]{Steel:89} that
\begin{eqnarray}
\alpha^\cutset_z 
&=& 
(q_x\alpha^\cutset_{x} + p_x\beta^\cutset_{x})(q_y\alpha^\cutset_{y}+p_y\beta^\cutset_{y})\nonumber\\
&& \qquad 
+ (q_x\alpha^\cutset_{x} + p_x\beta^\cutset_{x})(1-\alpha^\cutset_{y}-\beta^\cutset_{y})
+ (1-\alpha^\cutset_{x}-\beta^\cutset_{x})(q_y\alpha^\cutset_{y}+p_y\beta^\cutset_{y}),\label{eq:recurrence-alpha}
\end{eqnarray} 
where the first and second lines on the r.h.s.~correspond
respectively to cases (i) and (ii) above.
Similarly,
\begin{eqnarray}
\beta^\cutset_z &=& (p_x\alpha^\cutset_{x} + q_x\beta^\cutset_{x})(p_y\alpha^\cutset_{y}+q_y\beta^\cutset_{y})\nonumber\\
&& \qquad +(p_x\alpha^\cutset_{x} + q_x\beta^\cutset_{x})(1-\alpha^\cutset_{y}-\beta^\cutset_{y})
+ (1-\alpha^\cutset_{x}-\beta^\cutset_{x})(p_y\alpha^\cutset_{y}+q_y\beta^\cutset_{y}).\label{eq:recurrence-beta}
\end{eqnarray} 

In the case that
$p_u = p$ for all $u$, a fixed point analysis
was performed in~\cite[Theorem 7.22]{Steel:89}.
It was found that, if $p \geq 1/8$, there is a single
fixed point $(1/3,1/3)$ which corresponds informally to ``having
no information about the root.'' While
if $p < 1/8$, there is an additional
fixed point $(\alpha_p^\infty,\beta_p^\infty)$ with
$\alpha_p^\infty > \beta_p^\infty$. Convergence to 
$(1/3,1/3)$ in the first case and $(\alpha_p^\infty,\beta_p^\infty)$ in the second case was established rigorously
in~\cite{ZhShYa+:10}. One step in~\cite{ZhShYa+:10}
involved the derivation of a new recurrence in terms of the
quantities $\alpha_z^{\cutset_n} - \beta_z^{\cutset_n}$
and
$1 - (\alpha_z^{\cutset_n} + \beta_z^{\cutset_n})$,
which facilitates the analysis of the limit in the fixed edge weight case.

Going back to binary trees with general weights, as our starting point we further modify the recurrence 
of~\cite{ZhShYa+:10}. For all $z \in \vertices$,
we define
\begin{eqnarray}
\GG_z^\cutset 
=
\alpha_z ^\cutset - \beta_z^\cutset,
\qquad \text{and}\qquad
\FF_z^\cutset 
=
3(\alpha_z ^\cutset+\beta_z^\cutset) - 2.\label{eq:definition-f}
\end{eqnarray} 
We show in Proposition~\ref{prop:recurrence} below that $(\GG_z^\cutset,\FF_z^\cutset)$ satisfies the following recurrence
\begin{eqnarray}
	\GG^\cutset_z &=& \left(\frac{4-\FF^\cutset_y}{6}\right)\weight_x \GG^\cutset_x + \left(\frac{4-\FF^\cutset_x}{6}\right)\weight_y \GG^\cutset_y,\label{eq:recurrence-g}\\
	\FF^\cutset_z &=& \frac{3}{2}\weight_x\weight_y \GG^\cutset_x \GG^\cutset_y - \frac{1}{2}\FF^\cutset_x \FF^\cutset_y,\label{eq:recurrence-f}
\end{eqnarray} 
for $z \in \vertices^\cutset - \cutset$ with children
$x,y$,
as well as the inequalities 
$0 \leq \GG_z^\cutset \leq 1$ and $-1/2 \leq \FF_z^\cutset \leq 1$ for all $z \in \vertices^\cutset$ and the boundary conditions
$\GG_z^\cutset = \FF_z^\cutset = 1$ for all $z \in \cutset$.
Our choice of parametrization is motivated in part
by the fact that the ``no information'' fixed point is now
at $(0,0)$ and that $|\GG_z^\cutset|, |\FF_z^\cutset| \leq 1$.
At a high level we show that, under the branching rate condition~\eqref{eq:branching-cond}, the fixed point $(0,0)$ is ``unstable'' and that $\GG^\cutset_z$ in particular stays bounded away form $0$. That in terms implies a lower bound on the reconstruction accuracy as, by~\eqref{eq:ra-alphabeta}, we have
$$
\accuracy^\cutset_{\tree,\weight}
= \frac{1}{2} + \frac{1}{2}\GG^\cutset_\root.
$$ 

The link between stability and the weighted branching rate in~\eqref{eq:branching-cond} is easily seen from~\eqref{eq:recurrence-g}. First, when all weights are equal to $\minweight$ and $\cutset = \cutset_n$, we get by symmetry that close to $(0,0)$ to the first order
$
\GG^\cutset_z \approx \frac{4}{3}\minweight \GG^\cutset_x.
$
Hence, in that case, the solution can be expected to grow when $\frac{4}{3}\minweight > 1$, corresponding to the condition derived in Example~\ref{example:br-fixed-weights}. More generally, we use~\eqref{eq:recurrence-g} to relate the
$\GG$-value at the root to the $\GG$-values on a cutset (see Lemma~\ref{lem:root-cutset}). To deal with the nonlinear nature of~\eqref{eq:recurrence-g} and~\eqref{eq:recurrence-f}, we control
the $\FF$-values thanks to the quadratic form of~\eqref{eq:recurrence-f} which implies a quick decay towards $0$ (see Lemma~\ref{lem:controlling-f}).

\vspace{-0.4cm}\paragraph{Recurrence}
Before proceeding to the proof of our main results, 
we first establish a basic recurrence which follows from the work of~\cite{ZhShYa+:10}. We give a short proof for completeness. 
\begin{proposition}[Recurrence and basic properties]
	\label{prop:recurrence}
	The following hold:
	\begin{itemize}
		\item[-] {\em [Boundary conditions]} For all $z \in \cutset$, $\GG_z^\cutset = 1$ and $\FF_z^\cutset = 1$.
		
		\item[-] {\em [Recurrence]} For all $z \in \vertices^\cutset - \cutset$, if 
		$x,y$ are the children of $z$, the system~\eqref{eq:recurrence-g} and~\eqref{eq:recurrence-f} holds.
		
		\item[-] {\em [Bounds]} For all $z \in \vertices^\cutset$, we have $0 \leq \GG_z^\cutset \leq 1$ and $-1/2 \leq \FF_z^\cutset \leq 1$.
		
	\end{itemize}
\end{proposition}
\begin{proof}
	We start with the boundary conditions. By Definition~\ref{definition:fitch}(a) and the definitions
	of $\alpha_z^\cutset$ and $\beta_z^\cutset$,
	we have for all $z \in \cutset$ that
		$\alpha_z^\cutset = \P_{\tree,\weight}[\fitchstate^\cutset_z = \{\sigma_z\}] = 1$ 
		and 
		$\beta_z^\cutset = \P_{\tree,\weight}[\fitchstate^\cutset_z = \{1-\sigma_z\}] = 0$.
	So $\FF_z^\cutset = 3(\alpha_z ^\cutset+\beta_z^\cutset) - 2 =1$ and 
	$\GG_z^\cutset = \alpha_z ^\cutset - \beta_z^\cutset = 1$.

	The second statement is merely a change of variables. We briefly expand on the first equation (the other one being similar). Let $z \in \vertices^\cutset - \cutset$ with children
	$x,y$. We define $\Sigma_u = \alpha_u^\cutset + \beta_u^\cutset$ and $\Delta_u = \alpha_u^\cutset - \beta_u^\cutset$.
	By the definitions of $p_x$ and $q_x$, note that
	\begin{equation*}
	q_x\alpha^\cutset_{x} + p_x\beta^\cutset_{x}
	= \left(\frac{1+\weight_x}{2}\right)\alpha^\cutset_{x} 
	+ \left(\frac{1-\weight_x}{2}\right) \beta^\cutset_{x}
	= \frac{1}{2} \Sigma_x + \frac{1}{2}\weight_x \Delta_x,
	\end{equation*}
	and, similarly,
	$p_x\alpha^\cutset_{x} + q_x\beta^\cutset_{x}
	= \frac{1}{2} \Sigma_x - \frac{1}{2}\weight_x \Delta_x.$
	Hence, by~\eqref{eq:recurrence-alpha} and~\eqref{eq:recurrence-beta},
	\begin{equation*}
	\alpha_z^\cutset
	= \frac{1}{4}(\Sigma_x + \weight_x \Delta_x)
	(\Sigma_y + \weight_y \Delta_y)
	+ \frac{1}{2}(\Sigma_x + \weight_x \Delta_x)(1 - \Sigma_y)
	+ \frac{1}{2}(1-\Sigma_x)(\Sigma_y + \weight_y \Delta_y),
	\end{equation*}
	and
	\begin{equation*}
	\beta_z^\cutset
	= \frac{1}{4}(\Sigma_x - \weight_x \Delta_x)
	(\Sigma_y - \weight_y \Delta_y)
	+ \frac{1}{2}(\Sigma_x - \weight_x \Delta_x)(1 - \Sigma_y)
	+ \frac{1}{2}(1-\Sigma_x)(\Sigma_y  - \weight_y \Delta_y).
	\end{equation*}
	Subtracting the above two equations, we get
	\begin{eqnarray*}
		\Delta_z
		&=&  
		\frac{1}{2} \weight_x \Delta_x \Sigma_y 
		+ \frac{1}{2} \weight_y \Delta_y \Sigma_x
		+ \weight_x \Delta_x (1 - \Sigma_y)
		+ \weight_y \Delta_y (1-\Sigma_x)\\
		&=&  
		\left(1 - \frac{1}{2}\Sigma_y\right) \weight_x \Delta_x 
		+ \left(1 - \frac{1}{2}\Sigma_x\right) \weight_y \Delta_y,
	\end{eqnarray*}
	which after plugging 
	in~\eqref{eq:definition-f} gives~\eqref{eq:recurrence-g}.
	
	Because $\alpha_z^\cutset$ and $\beta_z^\cutset$ are
	probabilities and further $\alpha_z^\cutset + \beta_z^\cutset \leq 1$, we have that 
	$\GG_z^\cutset \leq 1$ 
	and
	$\FF_z^\cutset \leq 1$, 
	for all $z$.
	Moreover, that together with the boundary conditions 
	and~\eqref{eq:recurrence-g}
	implies that $\GG_z^\cutset \geq 0$ for all $z$ by induction. In turn, that together with the boundary conditions 
	and~\eqref{eq:recurrence-f} 
	implies that
	$\FF_z^\cutset \geq -1/2$ for all $z$.
\end{proof}

\section{Deterministic weights}

Before proceeding with the proof of Theorem~\ref{theorem:deterministic}, 
we first prove some lemmas. 

\subsection{Controlling $\GG$- and $\FF$-values}

In the first lemma, we express the $\GG$-value
at the root as a function of the $\GG$- and $\FF$-values above an arbitrary cutset. Recall that
$s(z)$ is the sibling of $z$.
\begin{lemma}[Controlling the root with a cutset]
\label{lem:root-cutset}
For any cutset $\cutset'$ in $\tree^\cutset$, it holds that
\begin{equation}
\GG^\cutset_{\rho} 
= \sum_{x\in\cutset'} \GG^\cutset_{x} \prod_{\root \neq z \leq x} \left\{\frac{4-\FF^\cutset_{s(z)}}{6}\right\}\theta_z. \label{GBranchingNumber}
\end{equation}
\end{lemma}
\begin{proof}
The result follows by recursively applying~\eqref{eq:recurrence-g} from the
root down to $\cutset'$. We implicitly use the fact that, by definition, 
a cutset is minimal. 
\end{proof}

Our second lemma shows that $\GG$-values cannot
grow too fast down the tree. This fact will be useful to proving the next key lemma. We will need the lower bound 
$
\minweight = \inf_{z \in \vertices} \weight_z > 0,
$
on the $\weight$-values. For $v,w \in \vertices^\cutset$, we let $\gamma(v,w)$
be the graph distance between $v$ and $w$ in $\tree^\cutset$. Recall that $\minweight \leq 1 < 2$.
\begin{lemma}[Growth of $\GG$-values]
\label{lem:g-growth}
Fix any $v \in \vertices^\cutset$. Then $\forall\eps' > 0$,
for all descendants $w$ of $v$ in $\vertices^\cutset$
$$
\GG_v^\cutset \leq \eps'
\implies 
\GG^\cutset_w \leq \eps' \left(\frac{2}{\minweight}\right)^{\gamma(v,w)}.
$$
\end{lemma}
\begin{proof}
Let $z \in \vertices^\cutset$, not on $\cutset$, have children $x$ and $y$. 
(Note that, in the case where $z$ is the parent of a vertex on the cutset $\cutset$, $z$ itself cannot be on the cutset by minimality and therefore both its children are in $\vertices^\cutset$.)
By Proposition~\ref{prop:recurrence}, 
we have $\FF^\cutset_x, \FF^\cutset_y \leq 1$
and $\GG^\cutset_x, \GG^\cutset_y \geq 0$, which implies that
both terms on the r.h.s.~of~\eqref{eq:recurrence-g} are non-negative.
Hence, using $\weight_x \geq \minweight$, \eqref{eq:recurrence-g} gives
$\GG^\cutset_z 
\geq \frac{1}{2} \minweight \GG^\cutset_x$.
In particular,  $\GG^\cutset_z \leq \eps'$ implies that
$\GG^\cutset_x \leq \eps' (2/\minweight)$. 
Recursing gives the claim.
\end{proof}

Our final lemma controls $\FF$-values at the root of a subtree where
$\GG$-values are uniformly small. 
\begin{lemma}[Controlling $\FF$ when $\GG$ is small]
\label{lem:controlling-f}
Fix any $0 <  \phi' \leq 1/9$ and $v \in \vertices^\cutset$. Then there exists $\eps' > 0$ depending only on $\minweight$ and $\phi'$ such that:
 $
 \GG_v^\cutset \leq \eps'(2/\minweight)
 $
 implies 
 $
 \left|\FF^\cutset_v\right| \leq 4 \phi'.
 $
\end{lemma}
\begin{proof}
Let $H$ be the smallest non-negative integer such that
\begin{equation}
\label{eq:def-H}
\left(\frac{1}{2}\right)^{-1+ 2^{H}} \leq \phi'.
\end{equation}
Define $\eps' > 0$ to be the largest positive real such that
\begin{equation}
\label{eq:epsp-1}
\eps' \left(\frac{2}{\minweight}\right)^{H+1} \leq 0.99
\qquad\text{and}\qquad
\frac{3}{2} \left[\eps' \left(\frac{2}{\minweight}\right)^{H+1} \right]^2 \leq \phi'.
\end{equation}
The rest of the proof proceeds in two steps: we derive a simplified recurrence for $\FF$-values and solve it.
	
	{\it 1. Simplified recurrence:}  Assume that
	$\GG_v^\cutset \leq \eps' (2/\minweight)$. 
	Let $w$ be a descendant of $v$ with graph distance
	$\gamma(v,w) \leq H$. Then, by Lemma~\ref{lem:g-growth}, 
	$
	\GG^\cutset_w \leq \eps' \left(2/\minweight\right)^{H+1} < 1,
	$
	where we used~\eqref{eq:epsp-1}. 
 This show, in particular, that all  descendants of $v$ in $\vertices$ within graph distance $H$ are in fact strictly above $\cutset$, because $\GG$-values are $1$ on the cutset $\cutset$.
	Moreover,
	by the recurrence~\eqref{eq:recurrence-f} and the inequality~\eqref{eq:epsp-1}, for any
	descendant $w_0$ of $v$ with children $w_1, w_2$ in $\vertices^\cutset$ that are within graph distance $H$ of $v$,
	we have
	\begin{equation}
	\label{eq:bound-f-abs}
	\left|\FF^\cutset_{w_0}\right|
	\leq \frac{3}{2}\theta_{w_1} \theta_{w_2} \GG^\cutset_{w_1} \GG^\cutset_{w_2}
	+ \frac{1}{2} \left|\FF^\cutset_{w_1}\right| \left|\FF^\cutset_{w_2}\right|
	\leq \phi' + \frac{1}{2} \left|\FF^\cutset_{w_1}\right| \left|\FF^\cutset_{w_2}\right|,
	\end{equation}
	where we used that $\theta$-values are $\leq 1$ and~\eqref{eq:epsp-1}. 
	Define
	$
	\FFabs_h
	= \sup\left\{
	\left|\FF^\cutset_w\right|
	\,:\, \text{$w \in \vertices^\cutset$, $v \leq w$ and $\gamma(v,w) = h$}
	\right\},
	$
	and $\FFabs_0 = \left|\FF^\cutset_v\right|$. By the remark above, the set in the previous display is non-empty for all $h = 0,\ldots,H-1$.
	Taking a supremum on both sides of~\eqref{eq:bound-f-abs}
	gives the recurrence
	\begin{equation}
	\label{eq:fabs-recurse}
	\FFabs_{h}
	\leq \phi'
	+
	\frac{1}{2}
	\FFabs_{h+1}^2, \qquad 0 \leq h\leq H-1.
	\end{equation}
	
	{\it 2. Solution:} We show by induction on $h$ (backwards from $H-1$) that 
	\begin{equation}
	\label{eq:fabs-induction}
	\FFabs_h 
	\leq 3\phi'
	+ 
	\left(\frac{1}{2}\right)^{-1+2^{H-h}}.
	\end{equation}
	For the base of the induction $h = H-1$, we have indeed that
	$$
	\FFabs_{H-1} 
	\leq \phi'
	+ 
	\frac{1}{2}
	\leq 3\phi'
	+ 
	\left(\frac{1}{2}\right)^{-1+2^{H-(H-1)}},
	$$
	where we used~\eqref{eq:fabs-recurse} and the fact that
	$\FF$-values are $\leq 1$ in absolute value.
	Assume the induction claim~\eqref{eq:fabs-induction} 
	holds for all $h' +1 \leq h \leq H-1$. We show it then holds for $h = h'$.
	Indeed, by~\eqref{eq:fabs-recurse} again,
	\begin{eqnarray*}
		\FFabs_{h'}
		&\leq& \phi'
		+
		\frac{1}{2}
		\FFabs_{h'+1}^2
		\leq
		\phi'
		+
		\frac{1}{2}
		\left[3\phi'
		+ 
		\left(\frac{1}{2}\right)^{-1+2^{H-(h' +1)}}\right]^2\\
		&\leq& 
		\phi'
		\left[
		1 + \frac{9 \phi'}{2} + 3 \left(\frac{1}{2}\right)^{-1+2^{H- (H-1)}}
		\right]
		+ 
		\left(\frac{1}{2}\right)^{2\{-1+2^{H-(h' +1)}\}+1}\\
		&\leq& 
		\phi'
		\left[
		\frac{5}{2} + \frac{9 \phi'}{2}
		\right]
		+ 
		\left(\frac{1}{2}\right)^{-1+2^{H-h'}}.
	\end{eqnarray*}
	Because by assumption $\phi' \leq 1/9$, the square bracket
	above is $\leq 3$. That concludes the induction. 
	
 By our choice of $H$ in~\eqref{eq:def-H}, that implies
$
\left|\FF^\cutset_v\right|
= \FFabs_0
\leq 3\phi' 
+ \left(1/2\right)^{-1+2^{H}}
\leq 4\phi'. 
$ 
\end{proof}

\subsection{Proof of main theorem}

\begin{proof}[Proof of Theorem~\ref{theorem:deterministic}]
Fix $\cutset \in \allcutsets(\tree)$ and assume that
$\minweight > 0$ and that~\eqref{eq:branching-cond} holds.
Then
there is 
$0 < \phi' \leq 1/9$ and $0 < \zeta < 1$ such that
\begin{equation}
\label{eq:phiprime-zeta}
\sum_{x\in\cutset'}
\prod_{\root \neq z \leq x} 
\left\{\frac{2}{3}(1-\phi')\right\}\weight_{z} \geq \zeta,
\end{equation}
for all cutsets $\cutset' \in \allcutsets(\tree)$. For this
value of $\phi'$,
let $\eps'$ be as in Lemma~\ref{lem:controlling-f}
and define 
$
\eps = \eps' \zeta < \eps'.
$
The proof proceeds by contradiction. Assume that
$\GG^\cutset_\root \leq \eps$. 
Let $\cutset'$ be the cutset
of those nodes closest to the root where the
$\GG$-values first cross above $\eps'$, i.e., formally
$
\cutset'
=
\{
x \in \vertices^\cutset
\,:\,
\text{$\GG^\cutset_x > \eps'$ and $\GG^\cutset_z \leq \eps',\ \forall z\leq x$}
\}.
$
Such a cutset (which is necessarily minimal) exists because
$\GG^\cutset_v = 1$ for all $v \in \cutset$ and $\eps' > \eps$. 
By Lemma~\ref{lem:g-growth},
for all $z$ on or above $\cutset'$, i.e.~such that $z \leq x$ for some $x \in \cutset'$, we have
$
\GG^\cutset_z \leq \eps' \frac{2}{\minweight}
$ and 
$
\GG^\cutset_{s(z)} \leq \eps' \frac{2}{\minweight},
$
since the immediate parent of $z$ (and $s(z)$) has $\GG$-value
$\leq \eps'$ by definition of $\cutset'$. By Lemma~\ref{lem:controlling-f},
we then have
\begin{equation}
\label{eq:f-bound-above-piprime}
\left|\FF^\cutset_z\right| \leq 4 \phi'
\quad\text{and}\quad
\left|\FF^\cutset_{s(z)}\right| \leq 4\phi'.
\end{equation}
By Lemma~\ref{lem:root-cutset},~\eqref{eq:phiprime-zeta} 
and~\eqref{eq:f-bound-above-piprime}, summing over $\cutset'$
\begin{eqnarray*}
\GG^\cutset_{\rho} 
&=& \sum_{x\in\cutset'} \GG^\cutset_{x} \prod_{\root \neq z \leq x} \left\{\frac{4-\FF^\cutset_{s(z)}}{6}\right\}\weight_z
> \sum_{x\in\cutset'} \eps' \prod_{\root \neq z \leq x} \left\{\frac{2}{3}(1-\phi')\right\}\weight_z
\geq \eps' \zeta
= \eps,
\end{eqnarray*}
which is a contradiction. 
\end{proof}

%
%

\section{I.i.d.~weights}

In this section, we prove our main result in the i.i.d.~weight case. 
Because there is no
lower bound on the weights, Theorem~\ref{theorem:deterministic}
cannot be applied directly to this case. In particular, the absence of a lower
bound makes controlling the $\FF$-values more challenging.
Here we identify a subtree of $\tree$ where $\FF$-values are well-behaved.
The existence of such a subtree is established with a coupling to a
percolation process, where open edges roughly indicate that weights
are uniformly bounded in a properly defined neighborhood.



\begin{proof}[Proof of Theorem~\ref{theorem:iid}]
First, we need the following percolation result.  
To each edge $e = (x,y)$ of $\tree$, 
where $x$ is the parent of $y$,
we assign an independent random weight $\tweight_y$ drawn from a distribution $\tdistr$ over $(0,1]$. We also pick an independent indicator
variable $\tperc_y$, which is $1$ with probability $\tilde{q} \in [0,1]$ and $0$ otherwise.
Let $\ttree = (\tvertices,\tedges)$ be
the subtree of $\tree$
whose vertices $x$ satisfy $\prod_{\root \neq z\leq x}\tperc_z = 1$.
We let $\noext$ be the event of non-extinction, i.e., the event that $\ttree$ is infinite. The following result can be proved along the lines of 
Proposition 3.2, Proposition 5.1 and Corollary 5.2 in~\cite{Peres:99}.
\begin{lemma}[Branching condition: random edge weights on open cluster]
	\label{lem:branching-random}
	Fix $\tilde{q} \in [0,1]$ and assume that $\tdistr$ has mean $\tilde{\mu} \in (0,1)$. Then, conditioned on 
	$\noext$, 
	almost surely
	\begin{equation}
	\sup\left\{\kappa > 0 : \inf_{\cutset \in \allcutsets(\ttree)}
	\sum_{x\in\cutset}\prod_{\root \neq z \leq x} \kappa^{-1}\tweight_{z} > 0\right\}
	= 2 \tilde{q} \tilde{\mu}.
	\end{equation}	
\end{lemma}
\noindent By standard branching process arguments~\cite{AtheryaNey:72}, the extinction probability $\tilde{\varphi}$
satisfies
$
\tilde{\varphi} 
= 
\tilde{q}^2
\tilde{\varphi}^2
+
2
(1-\tilde{q})
\tilde{q}
\tilde{\varphi}
+
(1-\tilde{q})^2,
$
i.e., 
\begin{equation}
\label{eq:tilde-extinct}
\P[\noext]
=
\left(\frac{1-\tilde{q}}{\tilde{q}}\right)^2.
\end{equation}

As discussed briefly above, we use a coupling argument.
In order to describe the coupling, we first need to define some constants (not depending on $\weight$).
Recall that $\mu$ is the mean of the edge weight distribution $\distr$ and
that $\delta$ is the desired failure probability.
Let $q \in (0,1)$ be close enough to $1$ that 
\begin{equation}
\label{eq:q-extinct}
\left(\frac{1-q}{q}\right)^2 \leq \frac{\delta}{3},
\end{equation}
and
$2 q \mu > 3/2$.
Let then $0 < \phi' \leq 1/9$
be such that 
\begin{equation}
\label{eq:q-phiprime}
2 q \mu > \frac{3}{2(1-\phi')}.
\end{equation}
Let $H$ be the smallest non-negative integer such that
$\left(1/2\right)^{-1+ 2^{H}} \leq \phi'$.
Because
$\P_{\weight \sim \distr}[\weight \leq \cutoff] \to 0$ as $\cutoff \to 0$, we can take
$\cutoff \in (0,1)$ small enough that 
$
\pcut
= \P[\weight \leq \cutoff],
$
satisfies
\begin{equation}
\label{eq:tau-q}
(1-\pcut)^{2^{H+1}} > q,
\qquad \text{and} \qquad
1-(1-\pcut)^{2^{H+2}} \leq \frac{\delta}{3}.
\end{equation}

We are now ready to define the coupled process $(\ttree,\tweight)$. We use the following notation: $\text{Ber($\psi$)}$ is a Bernoulli random variable with success probability $\psi$ and $Z|\mathcal{A}$ is the random variable $Z$ conditioned on the event $\mathcal{A}$.
For each $z$, we generate 
$
\weight_z
= 
(1 - I_z)\weight_{z,0} 
+ 
I_z \weight_{z,1},
$
where $I_z \sim \text{Ber($\P[\weight_z > \minweight]$)}$,  
$\weight_{z,0} \sim \weight_z|\{\weight_z \leq \cutoff\}$, 
and $\weight_{z,1} \sim \weight_z|\{\weight_z > \cutoff\}$ are independent.
For $z \in \vertices$, let $\neighb_h(z)$ be the descendants of $z$ lying exactly $h$ levels below it, i.e.,
$
\neighb_h(z)
=
\left\{
w \in \vertices\,:\, \text{$z \leq w$ and $\gamma(z,w) = h$}
\right\}.
$
Define
$
\tilde{J}_z
= \prod_{w \in \neighb_{H+1}(z)} I_w,
$
and
$
\tweight_z = 
\weight_{z,1}.
$
Note that by construction the random variables 
$\tilde{J}_z \sim \text{Ber($(1-\pcut)^{2^{H+1}} $)}$ 
and $\tweight_z$, 
$z\in \vertices$,
are independent.
Let $\ttree$ be defined as in Lemma~\ref{lem:branching-random},
let 
$
\tilde{q} = (1-\pcut)^{2^{H+1}}
$ 
and let $\tilde{\mu}$ be the mean
of $\distr$ conditioned on being larger than $\minweight$.
By~\eqref{eq:tilde-extinct},~\eqref{eq:q-extinct}, and~\eqref{eq:tau-q},
$\P[\noext] \leq \frac{\delta}{3}$.

We 
apply Lemma~\ref{lem:branching-random} 
to $(\ttree,\tweight)$ to obtain a branching rate condition 
similar to~\eqref{eq:phiprime-zeta}. 
\begin{lemma}[Towards controlling $\GG$ at the root]
There is a deterministic (i.e., not depending on $\weight$) $0 < \zeta < 1$ such that,
on the event of non-extinction, with probability at least $1-\delta/3$
\begin{equation}
\label{eq:delta-2}
\inf_{\tilde{\cutset}' \in \allcutsets(\ttree)}\sum_{x\in\tilde{\cutset}'}
\prod_{\root \neq z \leq x} 
\left\{\frac{2}{3}(1-\phi')\right\}\tweight_{z} \geq \zeta.
\end{equation}
\end{lemma}
\begin{proof}
On the event of non-extinction, we have almost surely that
$$
\sup\left\{\kappa > 0 : \inf_{\cutset \in \allcutsets(\ttree)}
\sum_{x\in\cutset}\prod_{\root \neq z \leq x} \kappa^{-1}\tweight_{z} > 0\right\}
= 2\tilde{q} \tilde{\mu} 
> 2 q \mu > \frac{3}{2(1-\phi')},
$$
where we used the fact that $\tilde{\mu} > \mu$ as well as~\eqref{eq:q-phiprime} and~\eqref{eq:tau-q}.
In particular, 
we can choose a deterministic (i.e., not depending on $\weight$) $0 < \zeta < 1$ such that,
on the event of non-extinction, with probability at least $1-\delta/3$ the inequality in~\eqref{eq:delta-2} holds.
\end{proof}

The purpose of the coupling is to show that the argument
used in Lemma~\ref{lem:controlling-f} to control the $\FF$-values can be applied to the vertices in $\ttree$. This is stated in the next lemma. For the choice
of $\phi'$, $H$ and $\minweight$ above,
let $\eps'$ be chosen as in the proof of Lemma~\ref{lem:controlling-f},
i.e., the largest positive real such that
\begin{equation*}
\frac{3}{2} \left[\eps' \left(\frac{2}{\minweight}\right)^{H+1} \right]^2 \leq \phi',
\qquad \text{and} \qquad
\eps' \left(\frac{2}{\minweight}\right)^{H+1} \leq 0.99.
\end{equation*}
\begin{lemma}[Controlling $\FF$ on $\ttree$]
For all $z \in \vertices \cap \tvertices$
\begin{equation}
\label{eq:2b-reloaded}
\GG_z^\cutset \leq \eps'\frac{2}{\minweight}
\implies 
\left|\FF^\cutset_z\right| \leq 4 \phi'
\ \ \text{and}\ \ 
\GG_{s(z)}^\cutset \leq \eps'\frac{2}{\minweight}
\implies 
\left|\FF^\cutset_{s(z)}\right| \leq 4 \phi'.
\end{equation}
\end{lemma}
\begin{proof}
	First, observe that for any $z \in \vertices \cap \tvertices$
	it must be that 
	$\weight_w > \minweight$, $\forall w \in \cup_{h \leq H+1} \neighb_h(z)$.
	Indeed, note that the unique $v$ such that $v \leq z$ and $\gamma(v,z) = h$
	satisfies $v \in \ttree$ and 
	$
	\neighb_{H+1-h}(z) \subseteq \neighb_{H+1}(v)
	$.

As a result, repeating the proof of Lemma~\ref{lem:controlling-f}, it follows
	that for all $z \in \vertices \cap \tvertices$
	\begin{equation*}
	\GG_z^\cutset \leq \eps'\frac{2}{\minweight}
	\implies 
	\left|\FF^\cutset_z\right| \leq 4 \phi'
	\ \ \text{and}\ \ 
	\GG_{s(z)}^\cutset \leq \eps'\frac{2}{\minweight}
	\implies 
	\left|\FF^\cutset_{s(z)}\right| \leq 4 \phi',
	\end{equation*}
	where we used the fact that the parent of $z$ is in $\ttree$ and therefore has 
	all its descendants within $H+1$ level with $\weight$-values above $\minweight$. The latter in turn implies that $s(z)$ itself has 
	all its descendants within $H$ level with $\weight$-values above $\minweight$.
That concludes the proof of the lemma.
\end{proof}

The rest of the proof of Theorem~\ref{theorem:iid} is similar to that of Theorem~\ref{theorem:deterministic}, except
that we restrict the argument to $\ttree$.
Define 
$
\eps = \eps' \zeta < \eps'.
$
Let $\mathcal{H}$ be the event that $\weight_w > \minweight$ for all
$w$ such that $\gamma(\root,w) \leq H$. By~\eqref{eq:tau-q}, $\P[\mathcal{H}] \geq 1 - \delta/3$. Condition on $\mathcal{H}$, $\noext$ and~\eqref{eq:delta-2},
which jointly occur with probability at least $1-\delta$.
Assume by contradiction that
$\GG^\cutset_\root \leq \eps$
and let
$
\cutset'
=
\{
x \in \vertices^\cutset
\,:\,
\text{$\GG^\cutset_x > \eps'$ and $\GG^\cutset_z \leq \eps',\ \forall z\leq x$}
\}.
$
By Lemma~\ref{lem:g-growth},
for all $z\in\vertices\cap \tvertices$ such that $z \leq x$ for some $x \in \cutset'$, we have
$
\GG^\cutset_z \leq \eps' \frac{2}{\minweight}
$
and
$
\GG^\cutset_{s(z)} \leq \eps' \frac{2}{\minweight},
$
since the immediate parent of $z$ (and $s(z)$) has $\GG$-value
$\leq \eps'$ by definition of $\cutset'$ and $\weight$-values on $\ttree$
are above $\minweight$. 
By~\eqref{eq:2b-reloaded},
we then have
\begin{equation}
\label{eq:f-bound-above-piprime-iid}
\left|\FF^\cutset_z\right| \leq 4 \phi'
\quad\text{and}\quad
\left|\FF^\cutset_{s(z)}\right| \leq 4\phi',
\end{equation}
for all $z\in\vertices\cap \tvertices$ such that $z \leq x$ for some $x \in \cutset'$.
Similarly to Theorem~\ref{theorem:deterministic}, by Lemma~\ref{lem:root-cutset},~\eqref{eq:delta-2} 
and~\eqref{eq:f-bound-above-piprime-iid},
\begin{eqnarray*}
	\GG^\cutset_{\rho} 
	&=& 
	\sum_{x\in\cutset'} \GG^\cutset_{x} \prod_{\root \neq z \leq x} \left\{\frac{4-\FF^\cutset_{s(z)}}{6}\right\}\weight_z
	\geq
	\sum_{x\in\cutset'\cap \tvertices} \GG^\cutset_{x} \prod_{\root \neq z \leq x} \left\{\frac{4-\FF^\cutset_{s(z)}}{6}\right\}\tweight_z
	\geq \eps' \zeta
    =\eps,
\end{eqnarray*}
where we used again that 
$\weight$-values on $\ttree$
are above $\minweight$ 
as well as
the fact that all terms in the sum are non-negative by the general bounds on the 
$\GG$- and $\FF$-values. 
This is a contradiction. 
\end{proof}


\section*{Acknowledgments}

	We thank Mike Steel for helpful discussions and an anonymous reviewer for suggested improvements to a previous version of the manuscript.

\bibliographystyle{alpha}
\bibliography{my,thesis}

\end{document}